\documentclass[11pt,reqno]{amsart}

\usepackage[T1]{fontenc}
\usepackage[utf8]{inputenc}
\usepackage{lmodern}
\usepackage{microtype}
\usepackage{amsmath,amssymb,amscd,mathtools}
\usepackage{xcolor}
\usepackage[colorlinks=true,linkcolor=blue!55!black,
  citecolor=blue!55!black,urlcolor=blue!55!black]{hyperref}
\hypersetup{
  pdftitle={Day Convolution and Homotopy Groups on the Pointed Homotopy
    Category of CW Complexes},
  pdfauthor={A. S. Yedathore}}

\numberwithin{equation}{section}

\newtheorem{theorem}{Theorem}[section]
\newtheorem*{theoremA}{Theorem A}
\newtheorem{proposition}[theorem]{Proposition}
\newtheorem{lemma}[theorem]{Lemma}
\newtheorem{corollary}[theorem]{Corollary}
\newtheorem{remark}[theorem]{Remark}
\theoremstyle{definition}
\newtheorem{definition}[theorem]{Definition}

\newcommand{\Z}{\mathbb Z}
\newcommand{\Ab}{\mathbf{Ab}}
\newcommand{\Set}{\mathbf{Set}}
\newcommand{\CW}{\mathbf{CW}_*}
\newcommand{\hCW}{h\mathbf{CW}_*}
\newcommand{\cC}{\mathcal C}
\newcommand{\cF}{\mathcal F}
\newcommand{\id}{\operatorname{id}}
\newcommand{\Hom}{\operatorname{Hom}}
\newcommand{\Nat}{\operatorname{Nat}}
\newcommand{\FreeAb}[1]{\Z\{#1\}}
\newcommand{\sm}{\wedge}
\newcommand{\Day}{\mathbin{\star}}
\newcommand{\extsm}{\mathbin{\boxtimes}}

\title[Day Convolution and Homotopy Groups]
  {Day Convolution and Homotopy Groups on the Pointed Homotopy Category of CW Complexes}

\author{Abhiram Srinivas Yedathore}
\address{Department of Physical Sciences, Indian Institute of Science
  Education and Research Mohali, Knowledge City, Sector 81, SAS Nagar,
  Manauli PO 140306, Punjab, India}
\email{ms22093@iisermohali.ac.in}


\begin{document}

\begin{abstract}
We freely enrich the pointed homotopy category of CW complexes over
abelian groups and equip the resulting functor category with the enriched
Day convolution induced by the smash product.  After identifying the
convolution unit and recording the coherence maps, we compute the
convolution of the homotopy-group functors.  For integers $n,m\geq 2$ we
construct a natural isomorphism
\[
  \pi_n\Day\pi_m\cong\pi_{n+m}
\]
induced by the external smash product; equivalently, the external smash
product is the universal natural pairing of homotopy groups.  Under this
identification, the symmetry interchanging the two factors acts on
$\pi_{n+m}$ by multiplication by $(-1)^{nm}$.
\end{abstract}

\maketitle

\section{Introduction}\label{sec:intro}

Day convolution~\cite{Day1970} extends a monoidal structure on a category
$\cC$ to the category of functors out of $\cC$; it is, for instance, the
construction behind the smash products of diagram spectra~\cite{MMSS}.
In this paper we apply it to the homotopy-group functors on the pointed
homotopy category of CW complexes, with the monoidal structure given by
the smash product.

The hom-sets $[X,Y]_*$ of the pointed homotopy category $\hCW$ carry no
natural addition, so we first enrich $\hCW$ freely over abelian groups.
This gives an $\Ab$-category $\cC=\Z[\hCW]$ with the same objects, in
which $\cC(X,Y)$ is the free abelian group on $[X,Y]_*$; additive
functors $\cC\to\Ab$ are the same as ordinary functors $\hCW\to\Ab$.  The
smash product makes $\cC$ symmetric monoidal, and Day convolution makes
the category $\cF$ of additive functors $\cC\to\Ab$ symmetric monoidal as
well, with product
\[
 (F\Day G)(X)=\int^{A,B\in\cC}\cC(A\sm B,X)\otimes F(A)\otimes G(B)
\]
and unit $J=\cC(S^0,-)\cong\FreeAb{\pi_0(-)}$, which is not the constant
functor $\Z$.  Each homotopy-group functor $\pi_n$ is an object of $\cF$.

For $x\in\pi_n(A)$ and $y\in\pi_m(B)$, let
$x\extsm y\in\pi_{n+m}(A\sm B)$ denote their external smash product
(Lemma~\ref{lem:external-pairing}).  Our main result is the following;
it combines Theorem~\ref{thm:main} and
Corollary~\ref{cor:graded-symmetry}.

\begin{theoremA}
Let $n,m\geq2$.  There is an isomorphism
$\mu_{n,m}\colon\pi_n\Day\pi_m\to\pi_{n+m}$ in $\cF$, given on
generators by $[e_{[u]}\otimes x\otimes y]_{A,B}\mapsto u_*(x\extsm y)$.
Under these isomorphisms, the symmetry of Day convolution interchanging
the two factors acts on $\pi_{n+m}$ by multiplication by $(-1)^{nm}$.
\end{theoremA}

Theorem~A has a more intrinsic formulation.  By the universal property of
the coend and the Yoneda lemma, morphisms $\pi_n\Day\pi_m\to G$ in $\cF$
correspond to families of homomorphisms
$\pi_n(A)\otimes\pi_m(B)\to G(A\sm B)$ natural in $A$ and $B$, and
$\mu_{n,m}$ corresponds to the external smash product.  Theorem~A
therefore says that $\extsm$ is the universal natural pairing of homotopy
groups: for every $G\in\cF$, composition with $\extsm$ is an isomorphism
\[
 \cF(\pi_{n+m},G)\xrightarrow{\ \cong\ }
 \Nat_{A,B}\bigl(\pi_n(A)\otimes\pi_m(B),\,G(A\sm B)\bigr),
 \qquad \theta\longmapsto\theta_{A\sm B}\circ\extsm.
\]

The proof is short, and the theorem is close to formal, but it does not
follow formally from the representable case.  For the linearized
representables $R_n=\cC(S^n,-)$, the Yoneda lemma gives
$R_n\Day R_m\cong R_{n+m}$ (Proposition~\ref{prop:rep-convolution}), and
$\pi_n$ is a quotient of $R_n$.  Since Day convolution preserves colimits
in each variable, $\pi_n\Day\pi_m$ is a quotient of $R_{n+m}$; the
content of Theorem~A is that it is the right quotient, namely that the
additivity relations defining $\pi_n$ and $\pi_m$ generate exactly those
defining $\pi_{n+m}$.  This cannot be deduced from Yoneda alone: $\pi_n$
is not representable, nor even a retract of a representable functor
(Remark~\ref{rem:not-representable}).  In the proof, coend balancing
moves every element onto the universal example
$(S^n,S^m,\iota_n,\iota_m)$, where $\iota_k\in\pi_k(S^k)$ is the class of
the identity, and bilinearity of the external smash
product matches the relations on the two sides.
Corollary~\ref{cor:square} records the resulting comparison with the
representable case.

Sections~\ref{sec:enrichment}--\ref{sec:unit} and
Proposition~\ref{prop:rep-convolution} are standard: they are the free
$\Ab$-category on an ordinary category~\cite[\S2.5]{Kelly}, Day's theorem~\cite{Day1970}, and the convolution of representables.  We include
proofs, or sketches, for completeness and to fix notation.  The new
content is Theorem~A and its proof.

\emph{Outline.}  Section~\ref{sec:foundations} fixes conventions and
deals with size.  Section~\ref{sec:enrichment} constructs $\cC$,
Section~\ref{sec:coends} recalls coends and defines Day convolution, and
Section~\ref{sec:unit} identifies the unit and records the coherence
data.  Section~\ref{sec:representables} treats the linearized
representables and presents $\pi_n$ as a quotient of $R_n$.
Section~\ref{sec:main} proves the isomorphism of Theorem~A and compares
it with the representable case, and Section~\ref{sec:symmetry} computes
the symmetry.

\section{Foundations and conventions}\label{sec:foundations}

We work with two Grothendieck universes $\mathcal U\in\mathcal V$; this
is the universe-enlargement device of~\cite[\S2.6]{Kelly}.
Let $\CW^{\mathcal U}$ be the category of pointed $\mathcal U$-small CW
complexes whose basepoint is a $0$-cell.  We regard these spaces as
compactly generated weak Hausdorff spaces and form products and quotients
in that category.  Thus the compactly generated product gives the usual
product CW structure, the wedge is a subcomplex, and the quotient
\[
  X\sm Y=(X\times Y)/(X\vee Y)
\]
is again a pointed CW complex.  The basepoints are nondegenerate, and the
smash product preserves based homotopies.  It therefore descends to a
bifunctor on the pointed homotopy category
\[
  -\sm-\colon\hCW^{\mathcal U}\times\hCW^{\mathcal U}
  \longrightarrow\hCW^{\mathcal U}.
\]
With unit $S^0$, the standard associator $\alpha$, unitors $\lambda$ and
$\rho$, and twist $\tau$ make $\hCW^{\mathcal U}$ a symmetric monoidal
category; see~\cite[Chs.~5 and~8]{MayConcise} for the point-set
conventions.  We write $\tau_{X,Y}\colon X\sm Y\to Y\sm X$.

For $X,Y\in\hCW^{\mathcal U}$, write
\[
  [X,Y]_* = \hCW^{\mathcal U}(X,Y)
\]
for the set of based homotopy classes of based maps.  Each such set is
$\mathcal U$-small.  The collection of all $\mathcal U$-small CW complexes
is $\mathcal V$-small, so the coends below exist in $\Ab^{\mathcal V}$,
the category of $\mathcal V$-small abelian groups.  Accordingly, the
functors in the category $\cF$ of Section~\ref{sec:coends} take values in
$\Ab^{\mathcal V}$: even when $F$ and $G$ take $\mathcal U$-small values,
their convolution is a priori only $\mathcal V$-small.  We henceforth
suppress the universe superscripts and write $\Ab$ for $\Ab^{\mathcal V}$.

\section{The additive enrichment}\label{sec:enrichment}

The hom-set $[X,Y]_*$ has no canonical abelian-group structure for
arbitrary pointed CW complexes.  We therefore enrich the homotopy category
freely over abelian groups.

\begin{definition}
Let
\[
  \cC=\Z[\hCW].
\]
The objects of $\cC$ are those of $\hCW$, and
\[
  \cC(X,Y)=\FreeAb{[X,Y]_*}
\]
is the free abelian group on the set of based homotopy classes.  We denote
the basis element associated to a class $[f]$ by $e_{[f]}$.
\end{definition}

Composition is the homomorphism
\[
  \cC(Y,Z)\otimes_{\Z}\cC(X,Y)\longrightarrow\cC(X,Z)
\]
determined by
\[
  e_{[g]}\otimes e_{[f]}\longmapsto e_{[g\circ f]}.
\]
Equivalently,
\[
 \left(\sum_i n_i e_{[g_i]}\right)
 \circ
 \left(\sum_j m_j e_{[f_j]}\right)
 =\sum_{i,j}n_im_j e_{[g_i\circ f_j]}.
\]
The enriched identity is the homomorphism
\[
  \Z\longrightarrow\cC(X,X),\qquad 1\longmapsto e_{[\id_X]}.
\]
Associativity and the identity laws follow from the corresponding laws in
$\hCW$ and bilinearity.  To lighten notation, when a map $f$ appears
inside a composite in $\cC$ we often write $f$ for $e_{[f]}$; for
instance, $u\circ\tau_{B,A}$ means $u\circ e_{[\tau_{B,A}]}$ for
$u\in\cC(A\sm B,X)$.

The following is the universal property of the free $\Ab$-category on an
ordinary category~\cite[\S2.5]{Kelly}; we include the short proof.

\begin{proposition}\label{prop:linearization}
Restriction to basis morphisms induces an isomorphism between the category
of $\Ab$-enriched functors $\cC\to\Ab$ and the category of ordinary
functors $\hCW\to\Ab$.
\end{proposition}

\begin{proof}
If $F\colon\hCW\to\Ab$ is an ordinary functor, then for every $X,Y$ its
action on morphisms is a function
\[
  [X,Y]_*\longrightarrow\Hom_{\Z}(F(X),F(Y)).
\]
By the universal property of the free abelian group, it extends uniquely
to a homomorphism
\[
  \FreeAb{[X,Y]_*}\longrightarrow\Hom_{\Z}(F(X),F(Y)),
  \qquad
  \sum_i n_i e_{[f_i]}\longmapsto\sum_i n_iF([f_i]).
\]
Functoriality of $F$ gives compatibility with enriched composition and
identities.  Conversely, restricting an enriched functor to the basis
morphisms recovers an ordinary functor.  These constructions are mutually
inverse and also identify ordinary and enriched natural transformations.
\end{proof}

The smash product extends bilinearly to hom-objects by
\[
  \cC(X,X')\otimes\cC(Y,Y')\longrightarrow
  \cC(X\sm Y,X'\sm Y'),
  \qquad e_{[f]}\otimes e_{[g]}\longmapsto e_{[f\sm g]}.
\]
This is an enriched bifunctor.  Conceptually, $(\cC,\sm,S^0)$ is obtained
from $(\hCW,\sm,S^0)$ by change of base along the strong symmetric
monoidal free-abelian-group functor $\FreeAb{-}\colon\Set\to\Ab$~\cite[\S2.5]{Kelly}, whose structure isomorphisms are
\[
  \FreeAb{S\times T}\cong\FreeAb S\otimes_{\Z}\FreeAb T,
  \qquad \FreeAb{\{*\}}\cong\Z.
\]
The coherence maps of $\hCW$ give basis morphisms in $\cC$, so their
naturality and coherence equations remain valid after linearization.
Hence $(\cC,\sm,S^0)$ is an $\Ab$-enriched symmetric monoidal category,
with associator, unitors and twist $\alpha$, $\lambda$, $\rho$, $\tau$.

\section{Enriched coends and Day convolution}\label{sec:coends}

The material of this section and the next is well known and goes back to
Day~\cite{Day1970}; see~\cite[Chs.~2--3]{Kelly} and~\cite[Chs.~1--2]{Loregian} for enriched ends, coends, Fubini's theorem
and the Yoneda lemma.  We include proofs, or sketches, for completeness
and to fix notation.

The category $(\Ab,\otimes_{\Z},\Z)$ is cocomplete, and tensor product
preserves colimits separately in each variable because
\[
  M\otimes_{\Z}-\dashv\Hom_{\Z}(M,-).
\]

Let $T\colon\cC^{\mathrm{op}}\otimes\cC\to\Ab$ be enriched.  Its coend is
the coequalizer
\begin{equation}\label{eq:coend-coeq}
 \int^{A\in\cC}T(A,A)
 =\operatorname{coeq}\!\left(
 \bigoplus_{A,B}\cC(A,B)\otimes T(B,A)
 \rightrightarrows
 \bigoplus_A T(A,A)
 \right).
\end{equation}
For $f\in\cC(A,B)$ and $x\in T(B,A)$, the two arrows impose the balancing
relation
\[
  T(f,1_A)(x)\sim T(1_B,f)(x).
\]
If $[x]_A$ denotes the class of $x\in T(A,A)$, this becomes
\[
  [T(f,1_A)(x)]_A=[T(1_B,f)(x)]_B.
\]
The coend is universal among abelian groups receiving a dinatural family
of homomorphisms from the groups $T(A,A)$.

Let
\[
  \cF=[\cC,\Ab]_{\Ab}
\]
be the $\Ab$-enriched functor category.  By
Proposition~\ref{prop:linearization}, its objects can equivalently be
viewed as ordinary functors $\hCW\to\Ab$.

\begin{definition}\label{def:day}
For $F,G\in\cF$ and $X\in\cC$, define their Day convolution by
\begin{equation}\label{eq:day}
  (F\Day G)(X)=
  \int^{A,B\in\cC}
  \cC(A\sm B,X)\otimes F(A)\otimes G(B).
\end{equation}
\end{definition}

For $u\in\cC(A\sm B,X)$, $x\in F(A)$ and $y\in G(B)$, write
$[u\otimes x\otimes y]_{A,B}$ for the class of $u\otimes x\otimes y$ in
$(F\Day G)(X)$.  These classes are additive in each of $u$, $x$ and $y$,
so $(F\Day G)(X)$ is generated by the \emph{elementary representatives}
\[
  [e_{[u]}\otimes x\otimes y]_{A,B},
\]
where $u\colon A\sm B\to X$ is a map.  Concretely, $(F\Day G)(X)$ is the
quotient of $\bigoplus_{A,B}\cC(A\sm B,X)\otimes F(A)\otimes G(B)$ by the
subgroup generated by the balancing relations
\begin{align*}
 [e_{[u]}\otimes F(f)(x)\otimes y]_{A',B}
 &= [e_{[u\circ(f\sm1_B)]}\otimes x\otimes y]_{A,B},\\
 [e_{[u]}\otimes x\otimes G(g)(y)]_{A,B'}
 &= [e_{[u\circ(1_A\sm g)]}\otimes x\otimes y]_{A,B},
\end{align*}
for all maps $u$, $f\colon A\to A'$ and $g\colon B\to B'$ and all $x$,
$y$.  Indeed, the relations imposed by the coend are additive in the
morphism and compatible with composition, and every morphism of
$\cC\otimes\cC$ is a sum of composites
$(e_{[f]}\otimes1)\circ(1\otimes e_{[g]})$.

For $h\in\cC(X,X')$, postcomposition induces
\[
 (F\Day G)(h)[u\otimes x\otimes y]_{A,B}
 =[h\circ u\otimes x\otimes y]_{A,B}.
\]
This is well defined and makes $F\Day G$ an enriched functor.  Enriched
natural transformations $\eta\colon F\to F'$ and $\theta\colon G\to G'$
act by
\[
 (\eta\Day\theta)_X[u\otimes x\otimes y]_{A,B}
 =[u\otimes\eta_A(x)\otimes\theta_B(y)]_{A,B}.
\]
Thus $\Day\colon\cF\otimes\cF\to\cF$ is an enriched bifunctor.  This is
the covariant form of Day's construction~\cite{Day1970,Day1971}.

\section{The convolution unit and coherence}\label{sec:unit}

We use the enriched co-Yoneda isomorphism in both of its forms: for
enriched functors $F\colon\cC\to\Ab$ and
$P\colon\cC^{\mathrm{op}}\to\Ab$ and objects $X,Y\in\cC$,
\begin{align}
  \int^A\cC(A,X)\otimes F(A)&\xrightarrow{\ \cong\ }F(X),
  &[u\otimes x]_A&\longmapsto F(u)(x),\label{eq:coyoneda}\\
  \int^A P(A)\otimes\cC(Y,A)&\xrightarrow{\ \cong\ }P(Y),
  &[p\otimes v]_A&\longmapsto P(v)(p).\label{eq:coyoneda-contra}
\end{align}
The inverses send $x\in F(X)$ to $[1_X\otimes x]_X$ and $p\in P(Y)$ to
$[p\otimes1_Y]_Y$.  The second form is the first applied to
$\cC^{\mathrm{op}}$; it is the one needed whenever the variable of
integration sits in the target of a representable, as in $\cC(I,A)$ or
$\cC(S^n,A)$ below.

Let $I=S^0$ and define the covariant representable
\[
  J=\cC(I,-).
\]
Thus
\[
  J(X)=\cC(S^0,X)=\FreeAb{[S^0,X]_*}
       \cong\FreeAb{\pi_0(X)},
\]
where the free abelian group includes the basepoint component as a
generator.  In particular, $J(X)\cong\Z$ when $X$ is path connected; the
unit is not the constant functor at $\Z$ on all pointed CW complexes.

\begin{proposition}[Day {\cite{Day1970}}]\label{prop:unit}
There are enriched natural isomorphisms
\[
  \lambda^{\star}_F\colon J\Day F\xrightarrow{\ \cong\ }F,
  \qquad
  \rho^{\star}_F\colon F\Day J\xrightarrow{\ \cong\ }F.
\]
\end{proposition}

\begin{proof}
By Fubini's theorem we may integrate one variable at a time.  Applying~\eqref{eq:coyoneda-contra} in the variable carrying $J$, then the unitors
$\lambda$ and $\rho$ of $\cC$, and then~\eqref{eq:coyoneda} gives
\begin{align*}
 (J\Day F)(X)
 &\cong\int^B\cC(I\sm B,X)\otimes F(B)
 \cong\int^B\cC(B,X)\otimes F(B)
 \cong F(X),\\
 (F\Day J)(X)
 &\cong\int^A\cC(A\sm I,X)\otimes F(A)
 \cong\int^A\cC(A,X)\otimes F(A)
 \cong F(X).
\end{align*}
Explicitly, if $u\in\cC(A\sm B,X)$, $a\in\cC(I,A)$,
$b\in\cC(I,B)$, and $x\in F(B)$ or $x\in F(A)$ as appropriate, then
\begin{align*}
 \lambda^{\star}_{F,X}[u\otimes a\otimes x]_{A,B}
 &=F\bigl(u\circ(a\sm1_B)\circ\lambda_B^{-1}\bigr)(x),\\
 \rho^{\star}_{F,X}[u\otimes x\otimes b]_{A,B}
 &=F\bigl(u\circ(1_A\sm b)\circ\rho_A^{-1}\bigr)(x).
\end{align*}
The inverses are
\[
 x\longmapsto[\lambda_X\otimes1_I\otimes x]_{I,X},
 \qquad
 x\longmapsto[\rho_X\otimes x\otimes1_I]_{X,I}.
\]
\end{proof}

Fubini's theorem for enriched coends, preservation of colimits by
$\otimes$, and~\eqref{eq:coyoneda-contra} give natural flattening
isomorphisms
\begin{align*}
 ((F\Day G)\Day L)(X)
 &\cong\int^{A,B,D}
 \cC((A\sm B)\sm D,X)\otimes F(A)\otimes G(B)\otimes L(D),\\
 (F\Day(G\Day L))(X)
 &\cong\int^{A,B,D}
 \cC(A\sm(B\sm D),X)\otimes F(A)\otimes G(B)\otimes L(D).
\end{align*}
Precomposition with the inverse
\[
 \alpha^{-1}_{A,B,D}\colon A\sm(B\sm D)\longrightarrow(A\sm B)\sm D
\]
of the associator of $\cC$ therefore defines the convolution associator
\[
 \alpha^{\star}_{F,G,L}\colon(F\Day G)\Day L\xrightarrow{\ \cong\ }
 F\Day(G\Day L).
\]
After flattening, the pentagon and triangle identities reduce to the
corresponding coherence identities in $\cC$.

The convolution symmetry is
\begin{equation}\label{eq:day-symmetry}
 \tau^{\star}_{F,G,X}[u\otimes x\otimes y]_{A,B}
 =[u\circ\tau_{B,A}\otimes y\otimes x]_{B,A},
\end{equation}
where $\tau_{B,A}\colon B\sm A\to A\sm B$ is the twist of $\cC$.
Naturality of $\tau$ proves compatibility with the balancing relations.
Involutivity, the hexagon identities, and compatibility with the unit
reduce, after flattening, to the corresponding identities in $\cC$ and
$\Ab$.

\begin{theorem}[Day {\cite{Day1970}}]\label{thm:day-monoidal}
The tuple
\[
  (\cF,\Day,J,\alpha^{\star},\lambda^{\star},\rho^{\star},\tau^{\star})
\]
is an $\Ab$-enriched symmetric monoidal category.
\end{theorem}

\section{Linearized representables}\label{sec:representables}

For $k\geq0$, let $S^k=\mathbb R^k\cup\{\infty\}$ be the one-point
compactification of $\mathbb R^k$, based at $\infty$.  For $n,m\geq0$,
let
\[
  \kappa_{n,m}\colon S^n\sm S^m\xrightarrow{\ \cong\ }S^{n+m}
\]
be the homeomorphism obtained by one-point compactifying the
concatenation $\mathbb R^n\times\mathbb R^m=\mathbb R^{n+m}$,
$(v,w)\mapsto(v,w)$.  The isomorphism of Theorem~\ref{thm:main} exists
for any choice of homeomorphisms $\kappa_{n,m}$, but the sign in
Corollary~\ref{cor:graded-symmetry} depends on this choice.

For $n\geq0$, let
\[
  R_n(X)=\cC(S^n,X)=\FreeAb{[S^n,X]_*}.
\]
This is the free abelian group on the underlying set of homotopy classes.
It should not be confused with $\pi_n(X)$: for $n\geq2$, the comparison
map $q_n$ below is never injective, since it sends the basis element $e_0\neq0$ of
the zero class to $0$.

The following proposition is an instance of the standard fact that the
Yoneda embedding $A\mapsto\cC(A,-)$, from $\cC^{\mathrm{op}}$ to
$(\cF,\Day)$, is strong monoidal~\cite{Day1970}.

\begin{proposition}\label{prop:rep-convolution}
For all $n,m\geq0$ there is a natural isomorphism
\[
  R_n\Day R_m\cong R_{n+m}.
\]
\end{proposition}

\begin{proof}
Two applications of~\eqref{eq:coyoneda-contra} give
\begin{align*}
 (R_n\Day R_m)(X)
 &=\int^{A,B}\cC(A\sm B,X)\otimes\cC(S^n,A)
      \otimes\cC(S^m,B)\\
 &\cong\cC(S^n\sm S^m,X)
 \cong\cC(S^{n+m},X).
\end{align*}
Explicitly,
\[
 [u\otimes a\otimes b]_{A,B}
 \longmapsto
 u\circ(a\sm b)\circ\kappa_{n,m}^{-1},
\]
the composite being formed in $\cC$; for elements $a$ and $b$ that are not
basis elements, it is extended linearly.  The inverse sends the basis
element of a map $w\colon S^{n+m}\to X$ to
\[
 [e_{[w\circ\kappa_{n,m}]}\otimes e_{[1_{S^n}]}
   \otimes e_{[1_{S^m}]}]_{S^n,S^m}.
\]
\end{proof}

For $n\geq2$ and $x\in\pi_n(X)=[S^n,X]_*$, write $e_x$ for the
corresponding basis element of $R_n(X)$, and define a natural quotient
\[
 q_n\colon R_n\twoheadrightarrow\pi_n,
 \qquad q_{n,X}(e_x)=x.
\]
Its kernel is generated by the elements
\[
 e_{x+y}-e_x-e_y,\qquad x,y\in\pi_n(X).
\]
(The basis element $e_0$ lies in this subgroup, since
$e_{0+0}-e_0-e_0=-e_0$.)  Consequently,
\[
 \pi_n(X)\cong
 \frac{R_n(X)}{\langle e_{x+y}-e_x-e_y\rangle}.
\]

\begin{remark}\label{rem:not-representable}
The functor $\pi_n$ is not representable, nor even a retract of a
representable functor: a retract of $\cC(A,-)$ takes values that are
direct summands of free abelian groups, hence free, whereas
$\pi_n(K(\Z/2,n))=\Z/2$.  Thus Theorem~\ref{thm:main} below is not a
formal consequence of Proposition~\ref{prop:rep-convolution};
Corollary~\ref{cor:square} records how the two are related.
\end{remark}

\section{Convolution of homotopy-group functors}\label{sec:main}

We now compute the Day convolution of the genuine homotopy-group
functors.  The key input is the external smash pairing.

\begin{lemma}\label{lem:external-pairing}
Let $n,m\geq2$.  For pointed CW complexes $A$ and $B$ and classes
$x\in\pi_n(A)$ and $y\in\pi_m(B)$, represented by maps
$a\colon S^n\to A$ and $b\colon S^m\to B$, put
\[
 x\extsm y=
 \left[
 S^{n+m}\xrightarrow{\kappa_{n,m}^{-1}}S^n\sm S^m
 \xrightarrow{a\sm b}A\sm B
 \right]\in\pi_{n+m}(A\sm B).
\]
This defines a map $\pi_n(A)\times\pi_m(B)\to\pi_{n+m}(A\sm B)$ that is
natural in $A$ and $B$ and bilinear.  Hence it induces a natural
homomorphism
\[
 \pi_n(A)\otimes\pi_m(B)\longrightarrow\pi_{n+m}(A\sm B).
\]
\end{lemma}

\begin{proof}
The class $x\extsm y$ is independent of the representatives, and natural
in $A$ and $B$, because the smash product preserves based homotopies.

Let $p_n\colon S^n\to S^n\vee S^n$ be the pinch map defining addition in
$\pi_n$, and let $c_1,c_2\colon S^n\vee S^n\to S^n$ be the two collapse
maps, so that $c_j\circ p_n\simeq1$.  Put $d=n+m$.  Use $\kappa_{n,m}$
to identify $S^n\sm S^m$ with $S^d$, and hence
\[
 (S^n\vee S^n)\sm S^m=(S^n\sm S^m)\vee(S^n\sm S^m)
\]
with $S^d\vee S^d$.  Under these identifications $p_n\sm1$ becomes a map
$p'\colon S^d\to S^d\vee S^d$, and its composite with the $j$-th collapse
map is $(c_j\circ p_n)\sm1\simeq1$, so it has degree one.  Since
$S^d\vee S^d$ is $(d-1)$-connected and $d\geq2$, the Hurewicz theorem~\cite[Theorem~4.32]{Hatcher} gives
\[
 \pi_d(S^d\vee S^d)\cong H_d(S^d\vee S^d)\cong\Z\oplus\Z,
\]
and the two coordinates of the class of a map $S^d\to S^d\vee S^d$ are
the degrees of its composites with the collapse maps.  Hence $p'$ is
homotopic to the pinch map $p_d$ on $S^d$.  If $a$ and $a'$ represent
$x,x'\in\pi_n(A)$, then $a+a'$ is represented by
$\nabla\circ(a\vee a')\circ p_n$, so
\[
 (a+a')\sm b\simeq
 \nabla\circ\bigl((a\sm b)\vee(a'\sm b)\bigr)\circ(p_n\sm1).
\]
Precomposing with $\kappa_{n,m}^{-1}$ and using $p'\simeq p_d$ gives
\[
 (x+x')\extsm y
 =x\extsm y+x'\extsm y.
\]
The same argument applied to $1\sm p_m$ proves additivity in the second
variable.
\end{proof}

\begin{theorem}\label{thm:main}
For $n,m\geq2$, there is an isomorphism in $\cF$
\[
  \mu_{n,m}\colon\pi_n\Day\pi_m\xrightarrow{\ \cong\ }\pi_{n+m}.
\]
It is given on elementary representatives by
\begin{equation}\label{eq:mu}
 \mu_{n,m,X}[e_{[u]}\otimes x\otimes y]_{A,B}
 =u_*(x\extsm y).
\end{equation}
\end{theorem}

\begin{proof}
For $A,B,X\in\cC$, define
\[
 \Theta_{A,B,X}\colon\cC(A\sm B,X)\otimes\pi_n(A)\otimes\pi_m(B)
 \longrightarrow\pi_{n+m}(X)
\]
on generators by
\[
 \Theta_{A,B,X}(e_{[u]}\otimes x\otimes y)
 =u_*(x\extsm y).
\]
This is well defined: the elements $e_{[u]}$ form a basis of
$\cC(A\sm B,X)$, and for fixed $u$ the right-hand side is additive in $x$
and in $y$ by Lemma~\ref{lem:external-pairing}, because $u_*$ is a
homomorphism.  If $f\colon A\to A'$ is a map, naturality of the external
pairing gives
\begin{align*}
 \Theta_{A',B,X}(e_{[u]}\otimes f_*(x)\otimes y)
 &=u_*\bigl(f_*(x)\extsm y\bigr)\\
 &=u_*(f\sm1_B)_*(x\extsm y)\\
 &=\Theta_{A,B,X}
   (e_{[u\circ(f\sm1_B)]}\otimes x\otimes y).
\end{align*}
The same argument in the $B$-variable gives the second family of
balancing relations.  By the description of the relations in
Section~\ref{sec:coends}, it suffices to check them for maps $f$ and $g$,
so $\Theta$ is balanced.  The universal property of the coend therefore
yields the natural transformation $\mu_{n,m}$ in~\eqref{eq:mu}.

It remains to construct its inverse.  Let
\[
 \iota_n=[1_{S^n}]\in\pi_n(S^n),
 \qquad
 \iota_m=[1_{S^m}]\in\pi_m(S^m).
\]
Given $\gamma\in\pi_{n+m}(X)$, represented by a map
$w\colon S^{n+m}\to X$, put
\[
 \bar\gamma=w\circ\kappa_{n,m}\colon S^n\sm S^m\longrightarrow X
\]
and define
\begin{equation}\label{eq:nu}
 \nu_{n,m,X}(\gamma)=
 [e_{[\bar\gamma]}\otimes\iota_n\otimes\iota_m]_{S^n,S^m}.
\end{equation}

We verify additivity.  Given $\gamma,\delta\in\pi_{n+m}(X)$, let
$A=S^n\vee S^n$, let $i_1,i_2\colon S^n\to A$ be the two inclusions, and
choose
\[
 u\colon A\sm S^m\cong(S^n\sm S^m)\vee(S^n\sm S^m)\longrightarrow X
\]
whose restrictions represent $\bar\gamma$ and $\bar\delta$; that is,
$u\circ(i_1\sm1)$ represents $\bar\gamma$ and $u\circ(i_2\sm1)$
represents $\bar\delta$.  In $\pi_n(A)$ one has
\[
 [p_n]=[i_1]+[i_2].
\]
Linearity in the middle factor, followed by balancing along $i_1$ and
$i_2$, gives
\begin{align*}
 [e_{[u]}\otimes[p_n]\otimes\iota_m]_{A,S^m}
 &= [e_{[u]}\otimes[i_1]\otimes\iota_m]_{A,S^m}
   +[e_{[u]}\otimes[i_2]\otimes\iota_m]_{A,S^m}\\
 &=\nu_{n,m,X}(\gamma)+\nu_{n,m,X}(\delta).
\end{align*}
On the other hand, balancing along $p_n$ gives
\[
 [e_{[u]}\otimes[p_n]\otimes\iota_m]_{A,S^m}
 =[e_{[u\circ(p_n\sm1)]}\otimes\iota_n\otimes\iota_m]_{S^n,S^m}.
\]
By Lemma~\ref{lem:external-pairing},
\[
 \bigl[u\circ(p_n\sm1)\circ\kappa_{n,m}^{-1}\bigr]
 =u_*([p_n]\extsm\iota_m)
 =u_*([i_1]\extsm\iota_m)+u_*([i_2]\extsm\iota_m)
 =\gamma+\delta,
\]
so $u\circ(p_n\sm1)$ represents $\overline{\gamma+\delta}$, and the
right-hand side above is $\nu_{n,m,X}(\gamma+\delta)$.  Hence
$\nu_{n,m,X}$ is a homomorphism; it is plainly natural in $X$.

The definitions give
\[
 \mu_{n,m,X}\nu_{n,m,X}(\gamma)=\gamma.
\]
Conversely, if
$[e_{[u]}\otimes x\otimes y]_{A,B}$ is an elementary
representative, choose maps $a\colon S^n\to A$ and $b\colon S^m\to B$
representing $x$ and $y$.  Since
\[
 a_*(\iota_n)=x,
 \qquad b_*(\iota_m)=y,
\]
balancing first in $A$ and then in $B$ gives
\begin{align*}
 [e_{[u]}\otimes x\otimes y]_{A,B}
 &= [e_{[u\circ(a\sm1_B)]}\otimes\iota_n\otimes y]_{S^n,B}\\
 &= [e_{[u\circ(a\sm b)]}\otimes\iota_n\otimes\iota_m]
    _{S^n,S^m}.
\end{align*}
Since $\mu_{n,m,X}[e_{[u]}\otimes x\otimes y]_{A,B}=u_*(x\extsm y)$ is
represented by $u\circ(a\sm b)\circ\kappa_{n,m}^{-1}$, this is exactly
$\nu_{n,m,X}\mu_{n,m,X}
[e_{[u]}\otimes x\otimes y]_{A,B}$.
Elementary representatives generate $(\pi_n\Day\pi_m)(X)$, so
$\mu_{n,m}$ and $\nu_{n,m}$ are inverse natural transformations.
\end{proof}

The isomorphism $\mu_{n,m}$ is determined by the representable case.

\begin{corollary}\label{cor:square}
For $n,m\geq2$, the square
\[
\begin{CD}
 R_n\Day R_m @>{\cong}>> R_{n+m}\\
 @V{q_n\Day q_m}VV @VV{q_{n+m}}V\\
 \pi_n\Day\pi_m @>{\mu_{n,m}}>> \pi_{n+m}
\end{CD}
\]
commutes, where the top arrow is the isomorphism of
Proposition~\ref{prop:rep-convolution}, and $q_n\Day q_m$ is an
epimorphism.  Hence $\mu_{n,m}$ is the unique morphism making the square
commute.
\end{corollary}

\begin{proof}
Let $u\colon A\sm B\to X$, $a\colon S^n\to A$ and $b\colon S^m\to B$ be
maps.  Along the top and right, the element
$[e_{[u]}\otimes e_{[a]}\otimes e_{[b]}]_{A,B}$ maps to
$e_{[u\circ(a\sm b)\circ\kappa_{n,m}^{-1}]}$ and then to
$[u\circ(a\sm b)\circ\kappa_{n,m}^{-1}]=u_*([a]\extsm[b])$.  Along the
left and bottom, it maps to $[e_{[u]}\otimes[a]\otimes[b]]_{A,B}$ and
then to $u_*([a]\extsm[b])$.  Elements of this form generate
$(R_n\Day R_m)(X)$, so the square commutes.  Every elementary
representative $[e_{[u]}\otimes x\otimes y]_{A,B}$ of
$(\pi_n\Day\pi_m)(X)$ is the image of
$[e_{[u]}\otimes e_x\otimes e_y]_{A,B}$, so $q_n\Day q_m$ is surjective
at every $X$, hence an epimorphism in $\cF$.
\end{proof}

In other words, the isomorphism of
Proposition~\ref{prop:rep-convolution} descends to the quotients, and
Theorem~\ref{thm:main} says that the induced map is again an
isomorphism.

\section{Graded symmetry}\label{sec:symmetry}

The twist $\tau_{S^n,S^m}\colon S^n\sm S^m\to S^m\sm S^n$ has a degree
only after both ends are identified with $S^{n+m}$, so consider the
self-map
\[
 t_{n,m}=\kappa_{m,n}\circ\tau_{S^n,S^m}\circ\kappa_{n,m}^{-1}
 \colon S^{n+m}\longrightarrow S^{n+m}.
\]
With $\kappa$ as in Section~\ref{sec:representables}, $t_{n,m}$ is the
one-point compactification of the block permutation
\[
 \mathbb R^n\times\mathbb R^m
 \longrightarrow
 \mathbb R^m\times\mathbb R^n,
 \qquad (v,w)\longmapsto(w,v),
\]
regarded as a linear automorphism of $\mathbb R^{n+m}$.  The one-point
compactification of a linear automorphism of $\mathbb R^d$ has degree
equal to the sign of its determinant, so $t_{n,m}$ has degree
$(-1)^{nm}$.

Precomposition with a self-map $t$ of $S^d$ of degree $k$ multiplies
$\pi_d(X)$ by $k$.  Indeed, $\pi_d(S^d)\cong\Z$ via degree~\cite[Corollary~4.25]{Hatcher}, so $[t]=k\iota_d$; if $w$ represents
$\gamma\in\pi_d(X)$, then
\[
 [w\circ t]=w_*[t]=w_*(k\iota_d)=k\gamma,
\]
because $w_*$ is a homomorphism.

\begin{corollary}\label{cor:graded-symmetry}
For $n,m\geq2$, the diagram
\[
\begin{CD}
 \pi_n\Day\pi_m @>{\tau^{\star}_{\pi_n,\pi_m}}>> \pi_m\Day\pi_n\\
 @V{\mu_{n,m}}V{\cong}V @VV{\mu_{m,n}}V\\
 \pi_{n+m} @>>{(-1)^{nm}}> \pi_{n+m}
\end{CD}
\]
commutes.  Thus Day symmetry recovers the usual graded sign.
\end{corollary}

\begin{proof}
Let $u\colon A\sm B\to X$ be a map, and let $a\colon S^n\to A$ and
$b\colon S^m\to B$ represent $x\in\pi_n(A)$ and $y\in\pi_m(B)$.  By~\eqref{eq:day-symmetry} and~\eqref{eq:mu},
\begin{align*}
 \mu_{m,n}\tau^{\star}_{\pi_n,\pi_m}[e_{[u]}\otimes x\otimes y]_{A,B}
 &=\mu_{m,n}[e_{[u\circ\tau_{B,A}]}\otimes y\otimes x]_{B,A}\\
 &=\bigl[u\circ\tau_{B,A}\circ(b\sm a)\circ\kappa_{m,n}^{-1}\bigr].
\end{align*}
By naturality of the twist,
$\tau_{B,A}\circ(b\sm a)=(a\sm b)\circ\tau_{S^m,S^n}$, so this class is
represented by
\[
 u\circ(a\sm b)\circ\kappa_{n,m}^{-1}\circ t_{m,n}.
\]
This is the precomposite of a representative of
$\mu_{n,m}[e_{[u]}\otimes x\otimes y]_{A,B}=u_*(x\extsm y)$ with
$t_{m,n}$, which has degree $(-1)^{nm}$.  Hence
$\mu_{m,n}\tau^{\star}_{\pi_n,\pi_m}=(-1)^{nm}\mu_{n,m}$ on elementary
representatives, which generate.
\end{proof}

\begin{remark}\label{rem:pointwise}
Theorem~\ref{thm:main} is a functor-level statement.  It does not, by
itself, produce a pointwise multiplication
\[
  \pi_n(X)\otimes\pi_m(X)\longrightarrow\pi_{n+m}(X)
\]
for an arbitrary pointed space $X$.  Such a multiplication requires
additional structure, for example a map $X\sm X\to X$.
\end{remark}

\section*{Acknowledgments}

I thank Vasudev Shyam for introducing me to Day convolution; David Ayala
for reading this paper and for his encouraging comments on the result;
and Rahul K. Balaji for discussions on the topic and for tremendous help
with the structure of the project.  I also thank Dikshith Poojary for his
time and effort in discussing functors with me; Bhoris Dhanjal for
insightful discussions at the blackboards of the International Centre for
Theoretical Sciences (ICTS), Bengaluru; and Skanda Shivakumar for sharing
insights on the project.


\begin{thebibliography}{99}

\bibitem{Day1970}
B. Day,
\emph{On closed categories of functors},
Reports of the Midwest Category Seminar IV,
Lecture Notes in Mathematics, vol. 137,
Springer, Berlin, 1970, pp. 1--38,
\href{https://doi.org/10.1007/BFb0060438}{\nolinkurl{doi:10.1007/BFb0060438}}.

\bibitem{Day1971}
B. Day,
\emph{Construction of biclosed categories},
Abstract of Ph.D. thesis,
Bull. Austral. Math. Soc. \textbf{5} (1971), no. 1, 139--140,
\href{https://doi.org/10.1017/S0004972700046979}{\nolinkurl{doi:10.1017/S0004972700046979}}.

\bibitem{Hatcher}
A. Hatcher,
\emph{Algebraic Topology},
Cambridge University Press, Cambridge, 2002.

\bibitem{Kelly}
G. M. Kelly,
\emph{Basic Concepts of Enriched Category Theory},
London Mathematical Society Lecture Note Series, vol. 64,
Cambridge University Press, Cambridge, 1982;
reprinted in Repr. Theory Appl. Categ. \textbf{10} (2005), 1--136.

\bibitem{Loregian}
F. Loregian,
\emph{(Co)end Calculus},
London Mathematical Society Lecture Note Series, vol. 468,
Cambridge University Press, Cambridge, 2021,
\href{https://doi.org/10.1017/9781108778657}{\nolinkurl{doi:10.1017/9781108778657}};
\href{https://arxiv.org/abs/1501.02503}{\nolinkurl{arXiv:1501.02503}}.

\bibitem{MMSS}
M. A. Mandell, J. P. May, S. Schwede, and B. Shipley,
\emph{Model categories of diagram spectra},
Proc. London Math. Soc. (3) \textbf{82} (2001), no. 2, 441--512,
\href{https://doi.org/10.1112/S0024611501012692}{\nolinkurl{doi:10.1112/S0024611501012692}}.

\bibitem{MayConcise}
J. P. May,
\emph{A Concise Course in Algebraic Topology},
Chicago Lectures in Mathematics,
University of Chicago Press, Chicago, 1999.

\end{thebibliography}
\end{document}